\documentclass[a4paper,12pt]{amsart}
\usepackage{amssymb}
\usepackage{mathrsfs}
\usepackage{color}
\usepackage[all]{xy}

\usepackage{color}
\usepackage{fancybox}
\usepackage[pdftex]{graphicx}

\theoremstyle{plain}
\newtheorem{theorem}{Theorem}[section]
\newtheorem{prop}[theorem]{Proposition}

\newtheorem{lemma}[theorem]{Lemma}

\theoremstyle{definition}
\newtheorem{remark}[theorem]{Remark}

\newtheorem{defn}[theorem]{Definition}
\newtheorem{example}[theorem]{Example}

\renewcommand{\setminus}{\smallsetminus}
\newcommand{\R}{\mathbb R}

\newcommand{\II}{I\hspace{-.1em}I}
\newcommand{\III}{I\hspace{-.1em}I\hspace{-.1em}I}

\numberwithin{equation}{section}

\title[Maximum genus of the Jenga like configuration]{Maximum genus of the Jenga like configurations}
\author[R.~Akiyama, N.~Abe, H.~Fujita, Y.~Inaba, M.~Hataoka, S.~Ito, S.~Seita]{Rika Akiyama, Nozomi Abe, Hajime Fujita, Yukie Inaba, \\ Mari Hataoka, Shiori Ito, Satomi Seita, }


\address[R.~Akiyama]{Department of Mathematical and Physical Sciences, Japan Women's University, 2-8-1 Mejirodai, Bunkyo-ku, Tokyo 112-8681, Japan}
\address[N.~Abe]{Spiber Inc., 234-1 Mizukami Kakuganji Tsuruoka, Yamagata 997-0052, Japan}
\address[H.~Fujita]{Department of Mathematical and Physical Sciences, Japan Women's University, 2-8-1 Mejirodai, Bunkyo-ku Tokyo, 112-8681, Japan} \email{fujitah@fc.jwu.ac.jp}
\address[Y.~Inaba]{Department of Mathematical and Physical Sciences, Japan Women's University, 2-8-1 Mejirodai, Bunkyo-ku, Tokyo 112-8681, Japan}
\address[M.~Hataoka]{TechFirm, 3-20-2 Nishi-shinjuku, Shinjuku-ku, Tokyo, 163-1423, Japan}
\address[S.~Seita]{Department of Mathematical and Physical Sciences, Japan Women's University, 2-8-1 Mejirodai, Bunkyo-ku, Tokyo 112-8681, Japan}
\begin{document}

\maketitle

\begin{abstract}
We treat the boundary of the union of blocks in the Jenga game as a surface with a polyhedral structure and consider its genus. We generalize the game and determine the maximum genus among the configurations in the generalized game. 
\end{abstract}

\section{Introduction  - The Jenga game and its maximum genus -}
{\it Jenga} is a game of physical skill marketed by Hasbro \cite{Hasbro} in Europe and Takara Tomy \cite{Takara} in Japan. 
The game starts from building blocks packed in three columns and 18 levels. 
Here we quote its rules from Wikipedia~\cite{wiki}. 

\begin{quote}
Jenga is played with 54 wooden blocks. Each block is three times longer than its width.   $\cdots$ 
Moving in Jenga consists of taking one and only one block from any level (except the one below the incomplete top level) of the tower, and placing it on the topmost level to complete it. 
$\cdots$
The game ends when the tower falls, or if any piece falls from the tower other than the piece being knocked out to move to the top. The winner is the last person to successfully remove and place a block.
\end{quote}

In this paper, we treat the boundary of the union of blocks in the game as a surface with a polyhedral structure, which is called a {\it polyhedral closed surface}, and consider its genus. In the initial configuration of the game the genus of the surface is 0. As the game progresses, the configuration of the surface changes and its genus may increase. Based on this observation, one may ask the following question: 

\begin{center}
By how much does the genus in the game increase?
\end{center}
Namely, when the Jenga game with $k$ levels is started for a given natural number $k(\geq 2)$, how can the maximum genus of the surface be described in terms of $k$? Of course we assume that the players do not make any mistakes. 

We generalize the game and consider the same question for the generalized game. Let $n$ and $k$ be two natural numbers. 
We consider a game that starts from building blocks packed in $n$ columns and $k$ levels. We call the game an {\it $(n,k)$-Jenga game} or an {\it $(n,k)$-game} for short. Basically, we adopt the rule for the original game (three columns). The significant point to note in the rules quoted above is 
\begin{center}
\lq\lq except the one below the incomplete top level \rq\rq.
\end{center}
We call a configuration of blocks which can appear in the $(n,k)$-game under the above rule the {\it Jenga like configuration}.  
In this paper, we determine the Jenga like configuration in the $(n,k)$-game that gives the maximum genus and compute the maximum genus for given $n$ and $k$. 

\medskip

\noindent
{\bf Main Theorem.~(Theorem~\ref{g(n,k)odd} and \ref{g(n,k)even}, Proposition~\ref{g(Q)leqg(n,k)even} and \ref{g(Q)leqg(n,k)odd})} \ {\it For given $n$ and $k(\geq 2)$ the maximum genus is realized by the $(n,k)$-configuration and its genus is given by the formula
$$
g(n,k)=\displaystyle\begin{cases}
\displaystyle\frac{n(n-2)(k-2)}{2} \quad (n \ {\rm is \ even}) \\ 
\\
\displaystyle\frac{n(n-1)(k-2)}{2} \quad (n \ {\rm is \ odd}). 
\end{cases}
$$}
The definition of the $(n,k)$-configuration is given in Section~\ref{defofnkjenga}. 
We derive the formula using the Gauss-Bonnet formula for a polyhedral closed surface or Descartes' theorem. To show the maximality of $g(n,k)$ among the $(n,k)$-game, we consider an algorithm to deform the $(n,k)$-configuration into the given configuration without increasing the genus. 

\section{Preliminaries}
Let $X$ be a polyhedron in $\R^3$. Namely $X$ is a subset in $\R^3$ obtained by gluing finitely many convex polygons along their vertices or along their edges. 
Let $F(X)$, $E(X)$ and $V(X)$ be the sets of all faces, edges and vertices in $X$, respectively. In this paper, we use a surface with a polyhedral structure defined as follows:
\begin{defn}[Polyhedral closed surface]
Let $Q$ be a polyhedron in $\R^3$. If $Q$ satisfies the following two conditions, then $Q$ is called a {\it polyhedral closed surface}. 
\begin{itemize}
\item Each edge of $Q$ is an edge of exactly two faces of $Q$. 
\item For each vertex $v\in V(Q)$, the link ${\rm lk}(v)$ of $v$ is connected. Here the link ${\rm lk}(v)$ is defined by 
$$
{\rm lk}(v):=\{e\in E(Q) \ | e\in E(f) \ {\rm for \ some} \ f\in F(Q), \ v\in V(f) \ {\rm and}  \ v\notin V(e)\}. 
$$
\end{itemize}
\end{defn}

If $Q$ is a polyhedral closed surface, the genus of $Q$ is denoted by $g(Q)$. For each vertex $v\in V(Q)$, $\kappa(v)$ denotes the {\it angular defect} of $v$, that is, 
$$
\kappa(v):=2\pi -\sum_{f\in F(Q), v\in V(f)}[{\rm angle \ of} \ f \ {\rm  at} \ v]. 
$$

The following is the main tool for us. 

\begin{theorem}[Gauss-Bonnet formula for closed polyhedral surface, Descartes' theorem]\label{GB}
For a polyhedral closed surface, $Q$, the following equality holds: 
$$
\sum_{v\in V(Q)}\kappa(v)=2\pi(\#V(Q)-\#E(Q)+\#F(Q))=4\pi(1-g(Q)). 
$$
\end{theorem}

\noindent
See Chapter 1 in \cite{Masuda} for these topics for example. 

\section{Genus and angular defects in the Jenga game}
We first describe how to play the $(n,k)$-game.  The game is played with $n\times k$ wooden blocks, where the length of each block is $n$ times its width. The initial configuration has $k$-levels and each level has $n$ blocks without gaps. The game is played by taking one and only one block from any level (except the one below the incomplete topmost level) of the tower, and placing it on the topmost level. Using the rules quoted above, with emphasis on \lq\lq except the one below the incomplete top level \rq\rq, we have the following two fundamental observations. 
\begin{itemize}
\item The second level from the top always has $n$ blocks. 
\item The sum of the numbers of blocks on the third level from the top and on the top level is greater than or equal to $n$. 
\end{itemize} 
These facts are important for the {\it $(n,k)$-configuration} (Definition~\ref{defofnkjenga}). 

In the $(n,k)$-game, each configuration has a structure of a polyhedral closed surface, and the genus of the configuration can be defined canonically. For example, the configuration in Figure~\ref{g=3} in the $(5,5)$-game has genus 3. 

\begin{figure}[h]
\begin{center}
\includegraphics[scale=0.3]{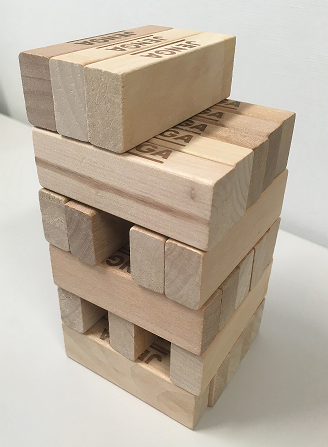}
\caption{A configuration with genus $3$ in the $(5,5)$-game} \label{g=3}
\end{center}
\end{figure}

The genus of a given configuration of the $(n,k)$-game can be computed by using Theorem~\ref{GB}. 
Essentially, following three types of vertices, Type $I$, Type $\II$ \ and Type $\III$ are needed. 

A Type $I$ \ vertex is a type of vertex that appears in the initial configuration of the $(3,4)$-game (see Figure~\ref{v123}~(1)). 
A Type $\II$ \ vertex is a type of vertex that newly appears when a block is removed from the second level of the initial configuration of the $(3,4)$-game (see Figure~\ref{v123}~(2)). 
A Type $\III$ \ vertex is a type of vertex that newly appears when a block is removed from the third level of the configuration described in the Type $\II$ \ vertex (see Figure~\ref{v123}~(3)). 

\begin{figure}[h]
\begin{center}
\includegraphics[scale=0.9]{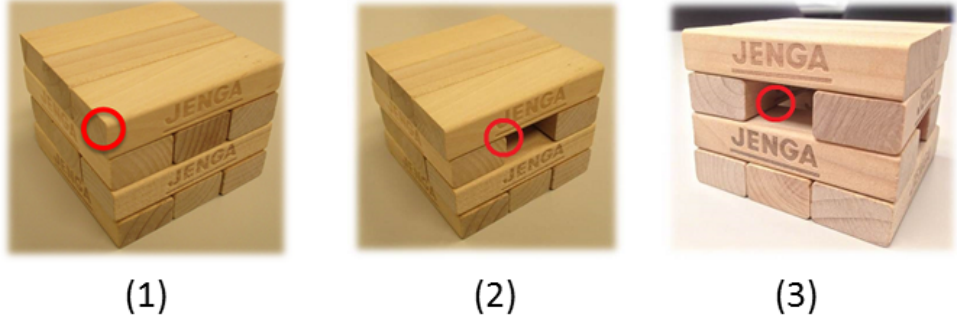}
\caption{Vertices of Type $I$, Type $\II$ \ and Type $\III$} \label{v123}
\end{center}
\end{figure}

These three types of vertices can be defined in a following rigorous way. 

\begin{defn}
A Type $I$ vertex is a vertex in a polyhedron in $\R^3$ whose neighborhood is isometric to a neighborhood of the origin of the region 
\[
\{(x,y,z)\in\R^3 \ | \ x\geq 0 \  {\rm and} \ y\geq 0 \ {\rm and} \ z\geq 0 \}. 
\]
A Type $\II$ vertex is a vertex in a polyhedron in $\R^3$ whose neighborhood is isometric to a neighborhood of the origin of the region 
\[
\{(x,y,z)\in\R^3 \ | \ [x\leq 0 \ {\rm or} \ z\leq 0] \ {\rm and} \ y\geq 0 \}. 
\]
A Type $\III$ vertex is a vertex in a polyhedron in $\R^3$ whose neighborhood is isometric to a neighborhood of the origin of the region 
\[
\{(x,y,z)\in\R^3 \ | \ [x\leq 0 \ {\rm and} \ z\geq 0] \ {\rm or} \ [y\geq 0 \ {\rm and} \ z\leq 0] \}. 
\]
\end{defn}

It can be seen that the Type $I$ vertex, $v_I$, has an angular defect 
\begin{equation}\label{kappa1}
\kappa(v_I)=2\pi - \frac{\pi}{2}\times 3=\frac{\pi}{2}. 
\end{equation}
Similarly, the angular defects of the Type $\II$ and Type $\III$ vertices, $v_{\II}$ and $v_{\III}$, respectively, are given by
\begin{equation}\label{kappa2}
\kappa(v_{I\hspace{-.1em}I})=2\pi - \left(\pi+\frac{\pi}{2}\times 3\right)=-\frac{\pi}{2}
\end{equation}and 
\begin{equation}\label{kappa3}
\kappa(v_{I\hspace{-.1em}I\hspace{-.1em}I})=2\pi - \left(\pi\times 2+\frac{\pi}{2}\times 2\right)=-\pi. 
\end{equation}

\section{Definition of the $(n,k)$-configuration}

Hereafter, the {\it box description} will be used to represent Jenga like configurations, as illustrated in Figure~\ref{example}. In this illustration, a gray or black square represents a block. Specifically, if a given configuration has $x$ levels, then in the box description, $n\times x$-cells are used, which are gray or black if the corresponding position contains a block. Note that each alternate levels have their perspective rotated by 90 degrees. 

\begin{figure}[h]
\begin{center}
\includegraphics[scale=0.5]{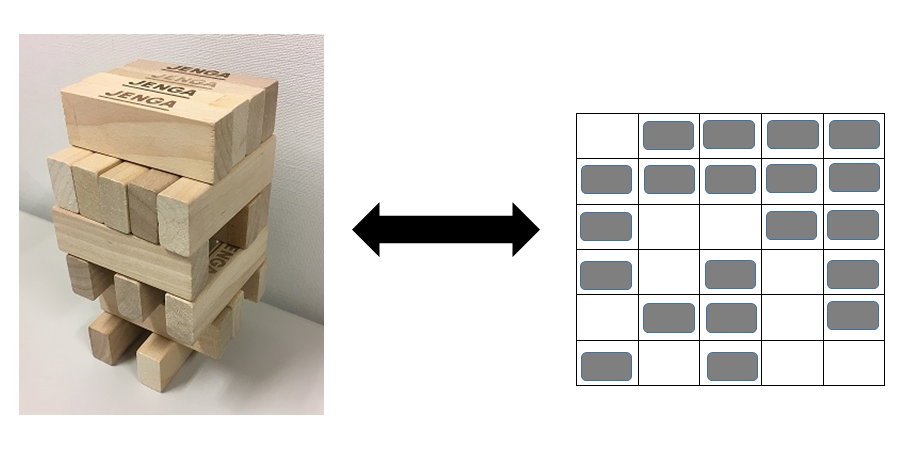}
\caption{Box description for a given configuration in the $(n,k)$-game.} \label{example}
\end{center}
\end{figure}

\begin{defn}[$(n,k)$-configuration]\label{defofnkjenga}
Let $n$ be an integer grater than 1 and  $k$ an integer grater than 2. We define {\it $(n,k)$-configuration} as follows. 
\\
(1) Suppose that $n$ is an odd integer. Then the {\it $(n,k)$-configuration} for odd case is the Jenga like configuration defined as follows (See Figure~\ref{nkodd} and Figure~\ref{nkoddpic}). 
\begin{itemize}
\item Let $x$ and $l$ be non-negative integers uniquely determined by the conditions 
\[
nk=n+\frac{n-1}{2}+\frac{n+1}{2}(x-3)+l  \ {\rm and} \  1\leq l \leq \frac{n-1}{2}. 
\]
\item It has $x$ levels. 
\item The top most level has $\frac{n-1}{2}$ blocks without gaps. 
\item The second level from the top has $n$ blocks. 
\item The bottom most level has $l$ blocks with at least one gap.
\item The rest of middle $x-3$ levels has $\frac{n+1}{2}$ blocks with one gap. 
\end{itemize}
(2) Suppose that $n$ is an even integer. Then the {\it $(n,k)$-configuration} for even case is the Jenga like configuration defined as follows (See Figure~\ref{nkeven} and Figure~\ref{nkevenpic}). 
\begin{itemize}
\item It has $2k-1$ levels. 
\item The top most level has $\frac{n}{2}$ blocks without gaps. 
\item The second top most level has $n$ blocks. 
\item The rest of $2k-3$ levels has $\frac{n}{2}$ blocks with at least one gap.
\end{itemize}

The polyhedral closed surface corresponding to the $(n,k)$-configuration is denoted by $Q(n,k)$,  and we also call $Q(n,k)$ the $(n,k)$-configuration. 

\end{defn}


\begin{figure}[h]
\begin{center}
\includegraphics[scale=0.5]{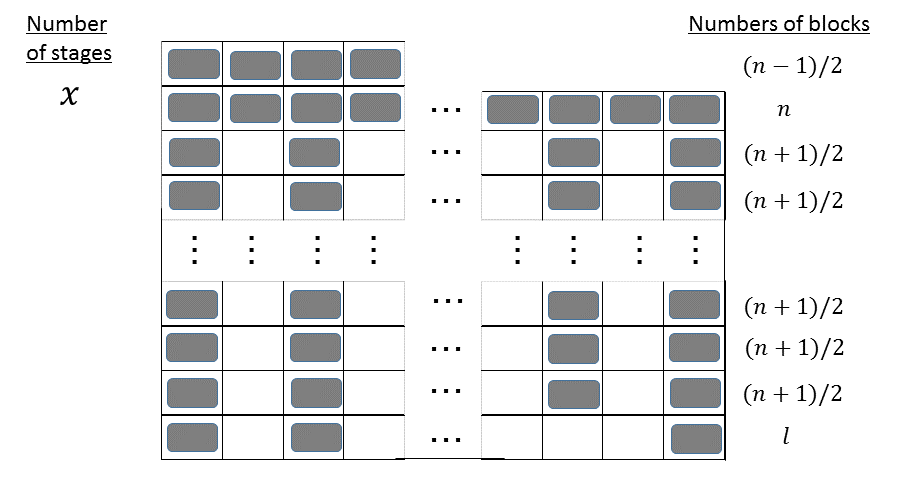}
\caption{$(n,k)$-configuration for odd $n$} \label{nkodd}
\end{center}
\end{figure}

\begin{figure}[h]
\begin{center}
\includegraphics[scale=1]{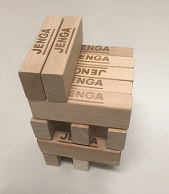}
\caption{$(5,3)$-configuration} \label{nkoddpic}
\end{center}
\end{figure}

\begin{figure}[h]
\begin{center}
\includegraphics[scale=0.5]{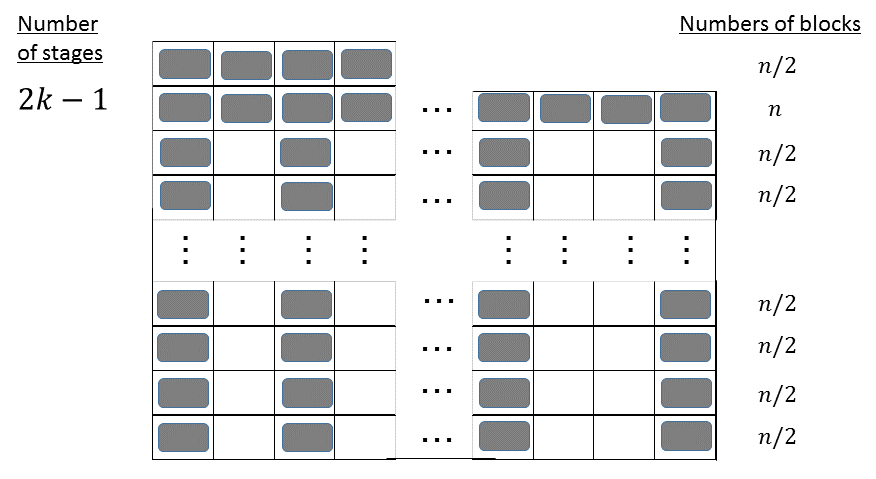}
\caption{$(n,k)$-configuration for even $n$} \label{nkeven}
\end{center}
\end{figure}

\begin{figure}[h]
\begin{center}
\includegraphics[scale=1]{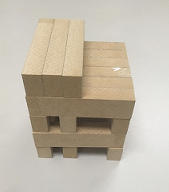}
\caption{$(6,3)$-configuration} \label{nkevenpic}
\end{center}
\end{figure}

\begin{remark}\label{bottommostodd}
In the $(n,k)$-configuration for odd case we impose the following condition for the configuration of the blocks in the bottom most level to count the vet rices of Type$I$, $\II$ and $\III$ in Proposition~\ref{NINIINIIIodd}. 
\begin{enumerate}
\item If $l=1$, then we put the block at the middle box, the $\frac{n+1}{2}$-th box from left (and right) in the box description.  
\item If $l=2$, then we put two blocks at the box described in (1) and the right most box. 
\item If $l\geq 3$, then we put blocks at the blocks described in (2) and the boxes from left with one gap. 
\end{enumerate}
\end{remark}

\begin{remark}
The $(n,k)$-configuration is in fact a Jenga like configuration. Namely when one starts the  $(n,k)$-game one can reach to the $(n,k)$-configuration by removing the blocks from the bottom most level in order. Moreover the $(n,k)$-configuration is physically stable\footnote{Of course we assume that players are prudence enough and they do not make any mistake. } under the condition in Remark~\ref{bottommostodd} for the bottom most level. 
\end{remark}


Let $N_I=N_I(n,k)$ be the number of Type $I$ vertices in the $(n,k)$-configuration. Analogously, $N_{\II}$ and $N_{\III}$ are the numbers of Type $\II$ and Type $\III$ vertices, respectively, in the $(n,k)$-configuration. Let $g(n,k)$ be the genus of the $(n,k)$-configuration, $Q(n,k)$. From Theorem~\ref{GB} and computation of the angular defects (\ref{kappa1}), (\ref{kappa2}) and (\ref{kappa3}), we have the following: 

\begin{lemma}\label{GB2}
The genus $g(n,k)$ of the $(n,k)$-configuration $Q(n,k)$ is given by 
$$
g(n,k)=-\frac{N_I}{8}+\frac{N_{\II}}{8}+\frac{N_{\III}}{4}+1. 
$$
\end{lemma}

\begin{remark}
In the subsequent sections we compute $N_I$, $N_{\II}$ and $N_{\III}$, however, in the computation we ignore the vertices of the topmost level because they do not contribute the genus $g(n,k)$ of $Q(n,k)$.  
\end{remark}

\section{Computation of the genus - odd case -}
In this section, we compute $N_I$, $N_{\II}$ and $N_{\III}$ and derive the formula for $g(n,k)$ when $n$ is odd under the condition in Remark~\ref{bottommostodd} . 

\begin{prop}\label{NINIINIIIodd}
If $n$ is odd and $l\geq 2$, then $N_I$, $N_{\II}$ and $N_{\III}$ are given by the following formulae: 
\begin{itemize}
\item[I.] $N_I=4+4l$
\item[\II.] $N_{\II}=4(x-3)(n-1)+4(l-1)$
\item[\III.] $N_{\III}=(x-4)(n-1)^2+2(l-1)(n-1)$
\end{itemize}
\end{prop}
\begin{proof}
We count the vertices of Type~$I$, $\II$ and $\III$ on each {\it floor},  where for each non-negative integer $i$, the $i$th floor is the intersection of the $i$th level and the ($i+1$)th level. For convenience, we call the intersection of the first level and the ground level the {\it $0$th floor}.

I. The formula for $N_I$ is clear. 

II. Let $N_{\II, i}$ be the number of Type~$\II$ \ vertices on the $i$th floor. It can be seen that (as seen in the cutaway of the $i$th floor in Figure~\ref{nii}) 
$$N_{\II,i}=
\begin{cases}
0 \qquad (i=0,x-1) \\ 
2(n-1)+4(l-1) \qquad (i=1) \\ 
2(n-1) \qquad (i=x-2) \\ 
4(n-1) \qquad (2\leq i \leq x-3),  
\end{cases}
$$and hence, we have 
$$
N_{\II}=\sum_iN_{\II, i}=4(x-3)(n-1)+4(l-1). 
$$

\begin{figure}[h]
\begin{center}
\includegraphics[scale=0.4]{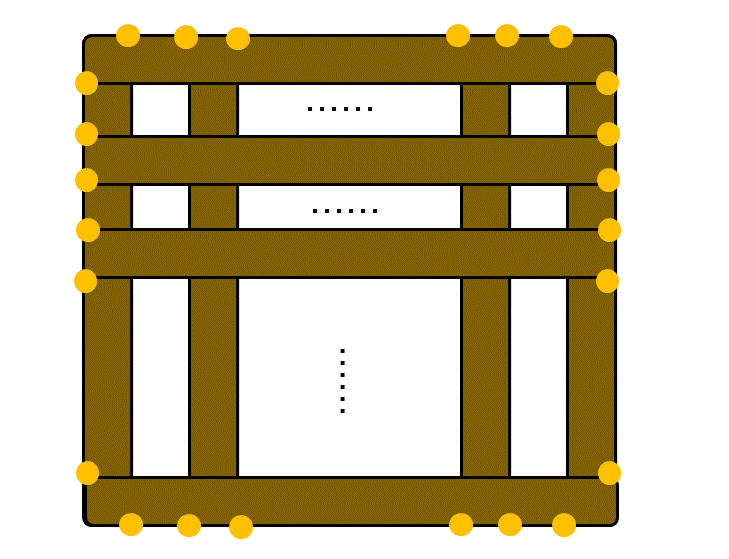}
\caption{Configuration of Typer$\II$ vertices in the  $i$th floor for $i=1,2,\ldots, x-3$ (Yellow circle = Type$\II$ vertex). } \label{nii}
\end{center}
\end{figure}


III. Let $N_{\III, i}$ be the number of Type ~$\III$ \ vertices on the $i$th floor. It can be seen that (as seen in the cutaway of the $i$th floor in Figure~\ref{niii}) 
$$N_{\III,i}=
\begin{cases}
0 \qquad (i=0,x-1, x-2) \\ 
2(l-1)(n-1) \qquad (i=1) \\ 
(n-1)^2 \qquad (2\leq i \leq x-3),  
\end{cases}
$$and, hence, we have 
$$
N_{\III}=\sum_iN_{\III, i}=(x-4)(n-1)^2+2(l-1)(n-1). 
$$

\begin{figure}[h]
\begin{center}
\includegraphics[scale=0.4]{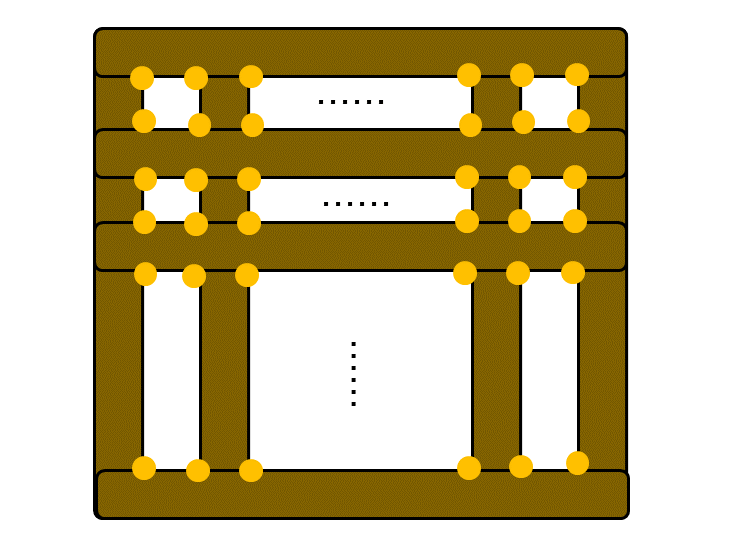}
\caption{Configuration of Typer$\III$ vertices in the  $i$th floor for $i=1,2,\ldots, x-3$ (Yellow circle = Type$\III$ vertex). } \label{niii}
\end{center}
\end{figure}

\end{proof}

\begin{remark}
The formulae in Proposition~\ref{NINIINIIIodd} are not correct for $l=1$. We use a trick for $l=1$ to resolve the case $l\geq 2$ in the proof of Theorem~\ref{g(n,k)odd}.  
\end{remark}

\begin{theorem}\label{g(n,k)odd}
When $n$ is odd, $g(n,k)$ is given by 
$$
g(n,k)=\frac{n(n-1)(k-2)}{2}. 
$$
\end{theorem}
\begin{proof}
We first prove this for the case $l\geq 2$ of the $(n,k)$-configuration with odd $k$. By Lemma~\ref{GB2} and Proposition~\ref{NINIINIIIodd}, we have 
\begin{equation}\label{karisome}
g(n,k)=\frac{(n^2-1)(x-4)+2l(n-1)}{4}. 
\end{equation}
By counting the number of blocks, we have 
\begin{equation}\label{relation}
nk=\frac{n-1}{2}+n+\frac{(x-3)(n+1)}{2}+l=\frac{(n+1)x}{2}+l-2. 
\end{equation}
By substituting (\ref{relation}) into (\ref{karisome}) we have 
$\displaystyle g(n,k)=\frac{n(n-1)(k-2)}{2}$. 

Now we consider the case $l=1$. In this argument, we deviate from rules of the game for a while. We deform the configuration by moving a block within the first level as shown in Figure~\ref{moving}.  
\begin{figure}[h]
\begin{center}
\includegraphics[scale=0.5]{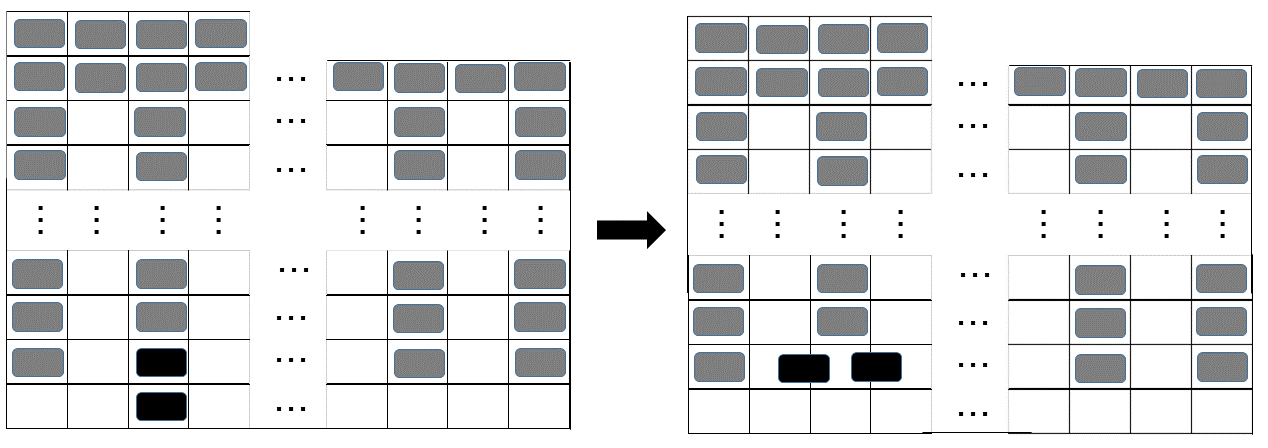}
\caption{A configuration with $l=1$ with $x$ levels and a configuration with $l=\frac{n+1}{2}+1$ with $x-1$ levels} \label{moving}
\end{center}
\end{figure}
Note that this operation does not change the genera of two configurations in Figure~\ref{moving}. 
By substituting $l\rightarrow\displaystyle \frac{n+1}{2}+1$ and $x\rightarrow x-1$ in Equation~(\ref{karisome}) the genus of the right configuration can be computed as
$$
\frac{1}{4}(n^2x-4n^2-x+2n+2), 
$$which is equal to Equation~(\ref{karisome}) with $l=1$. 
\end{proof}

\section{Computation of the genus - even case -}
In this section, we compute $N_I$, $N_{\II}$ and $N_{\III}$, and derive the formula for $g(n,k)$ when $n$ is even. 
The following can be proved in almost same way as Proposition~\ref{NINIINIIIodd}. 
\begin{prop}\label{NINIINIIIeven}
If $n$ is even, then $N_I$, $N_{\II}$ and $N_{\III}$ are given by the following formulae: 
\begin{itemize}
\item[I.] $N_I=4+2n$
\item[\II.] $N_{\II}=8(n-2)(4k-7)$
\item[\III.] $N_{\III}=2(n-2)^2(k-2)$
\end{itemize}
\end{prop}

\begin{theorem}\label{g(n,k)even}
When $n$ is even, $g(n,k)$ is given by 
$$
g(n,k)=\frac{n(n-2)(k-2)}{2}. 
$$
\end{theorem}
\begin{proof}
The formula can be obtained by Proposition~\ref{GB2} and Proposition~\ref{NINIINIIIeven}. 
\end{proof}

\section{The maximality of $g(n,k)$}
In this section we show that the genus $g(n,k)$ of the $(n,k)$-configuration is the maximum genus among all genera appearing in the $(n,k)$-game. Specifically, for the given configuration $Q$, we show that $g(Q)\leq g(n,k)$. To show it, we provide an algorithm for deforming the $(n,k)$-configuration $Q(n,k)$ into $Q$ without increasing the genus. 
In each step of the algorithm we deviate from the rules of the game. Namely we may treat a configuration which is not a Jenga like configuration and use operations which are forbidden in our rules. 

We first define the following three fundamental operations. 

\begin{itemize}
\item[(S)] Sliding a block within a level (see Figure~\ref{slide}).
\item[(L)] Removing a block and loading it onto the topmost level (see Figure~\ref{load}).
\item[(I)] Removing a block and inserting it into any other level (see Figure~\ref{insert}).
\end{itemize}

\begin{figure}[h]
\begin{center}
\includegraphics[scale=0.6]{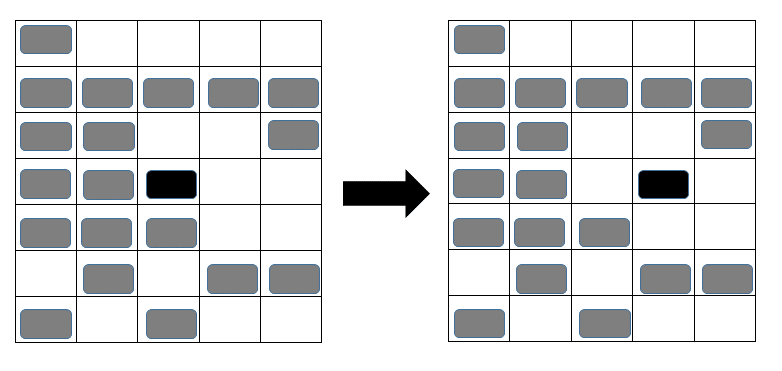}
\caption{The operation (S) performed on the black block} \label{slide}
\end{center}
\end{figure}

\begin{figure}[h]
\begin{center}
\includegraphics[scale=0.6]{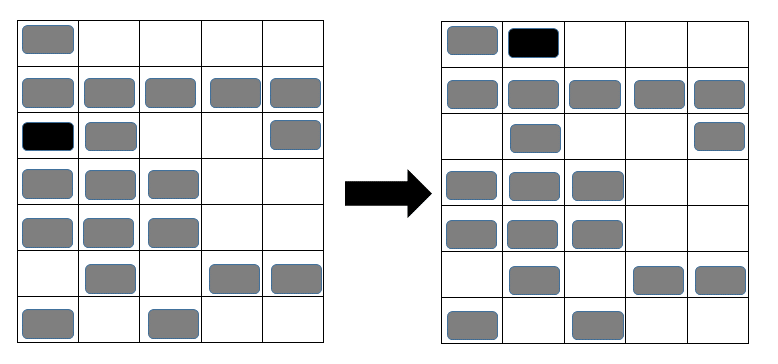}
\caption{The operation (L) performed on the black block} \label{load}
\end{center}
\end{figure}

\begin{figure}[h]
\begin{center}
\includegraphics[scale=0.6]{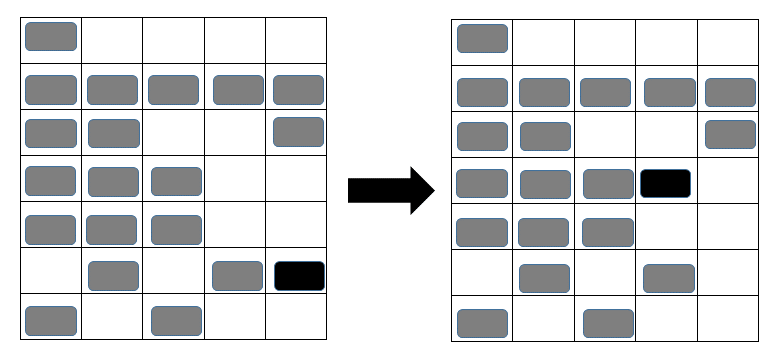}
\caption{The operation (I) performed on the black block} \label{insert}
\end{center}
\end{figure}

For the given configuration $Q$, the number of levels in $Q$ is denoted by $s(Q)$. 

\subsection{Proof for even $n$}
\label{Proof for even n}

Let $Q'(n,k)$ be the configuration obtained by removing the upper two levels and the first level from $Q(n,k)$. Similarly we consider the configuration $Q'$ removing from $Q$. See Figure~\ref{q}. 
 
\begin{figure}[h]
\begin{center}
\includegraphics[scale=0.6]{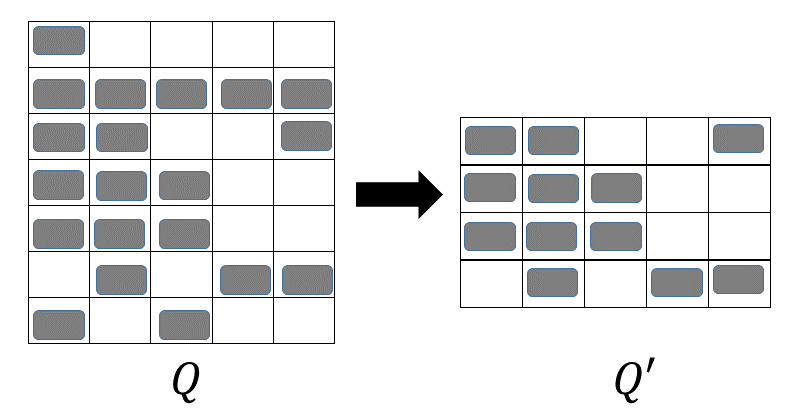}
\caption{$Q$ and $Q'$} \label{q}
\end{center}
\end{figure}

We first deform the configuration $Q'(n,k)$ into the configuration $Q'$ by applying (S), (L) and (I) a finite number of times
by the following algorithm. 
\begin{itemize}
\item[(A1)] For the $s$th level of $Q'(n,k)$ with $1\leq s\leq s(Q')$, whose number of blocks is greater than or equal to that of the $s$th level of $Q'$, we apply sufficiently many (L) and (S)  so that the resulting configuration of the $s$th level is the same as that of $Q'$. We apply these operations for all conceivable values of $s$. 
Let $Q_1'(n,k)$ be the configuration obtained by the above operation for $Q'(n,k)$. 
\item[(A2)] We apply the same operations for all conceivable $s$th level of $Q_1'(n,k)$ with $s(Q'(n,k))+1\leq s \leq s(Q')$. 
Let $Q'_2(n,k)$ be the configuration obtained by the above operation for $Q'_1(n,k)$. 
\item[(A3)] For the $s$th level of $Q'_2(n,k)$ with $1\leq s\leq s(Q')$, whose number of blocks is less than that of $s$th level of $Q'$, we apply finitely many (I) and (S) operations by using blocks between the $(s(Q')+1)$th level and the $s(Q_2'(n,k))$th level of $Q'_2(n,k)$, so that the resulting configuration of the $s$th level is the same as that of $Q'$. We apply this operation for all conceivable values of $s$. If the blocks in $Q_2'(n,k)$ become insufficient, then we may use blocks in $Q(n,k)\setminus Q'(n,k)$. 
Let $Q'_3(n,k)$ be the configuration obtained by the above operations for $Q'_2(n,k)$. 
\end{itemize}

\begin{example}
Here we demonstrate the algorithm by using an example. Consider the $(6,5)$-game. Let $Q$ be a configuration appearing in the game with the associated configuration $Q'$, as shown in Figure~\ref{Qex}.

\begin{figure}[h]
\begin{center}
\includegraphics[scale=0.6]{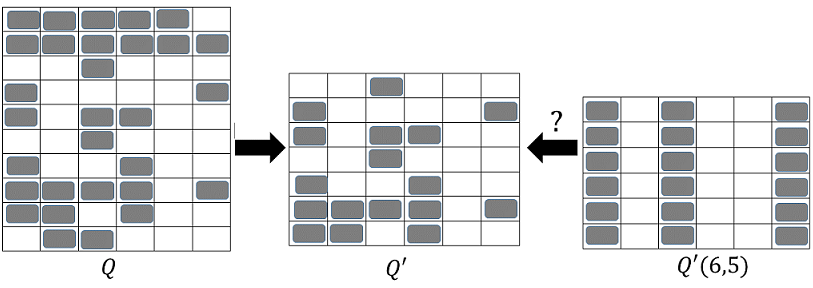}
\caption{$Q$, $Q'$ and $Q'(6,5)$} 
\label{Qex}
\end{center}
\end{figure}

We apply step (A1) of the algorithm to the first, third, fourth, fifth and sixth levels in $Q'(6,5)$ to obtain $Q_1'(6,5)$ (see Figure~\ref{Qex2}).


\begin{figure}[h]
\begin{center}
\includegraphics[scale=0.3]{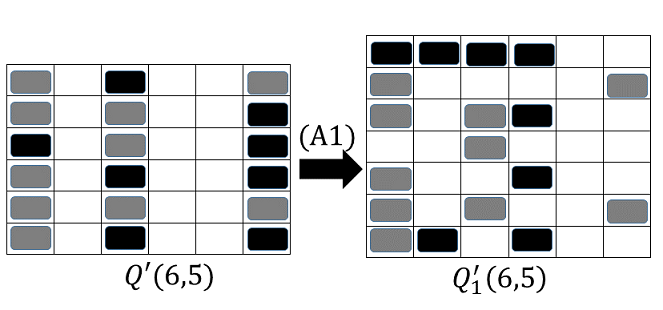}
\caption{$Q'(6,5)$ and $Q'_1(6,5)$} \label{Qex2}
\end{center}
\end{figure}

Next, we apply step (A2) to the seventh level in $Q_1'(6,5)$ to obtain $Q_2'(6,5)$, as shown in Figure~\ref{Qex3}. 

\begin{figure}[h]
\begin{center}
\includegraphics[scale=0.65]{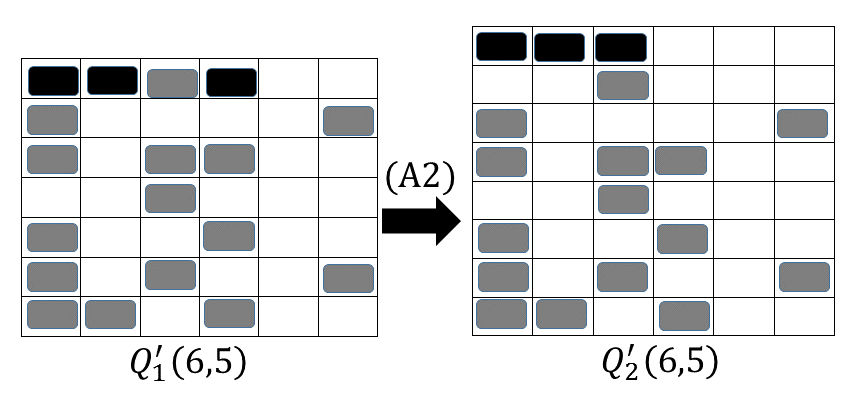}
\caption{$Q'_1(6,5)$ and $Q'_2(6,5)$} \label{Qex3}
\end{center}
\end{figure}

Finally, we apply step (A3) of the algorithm to the second level in $Q_2(6,5)$ to obtain $Q_3'(6,5)$ (see Figure~\ref{Qex4}). 

\begin{figure}[h]
\begin{center}
\includegraphics[scale=0.65]{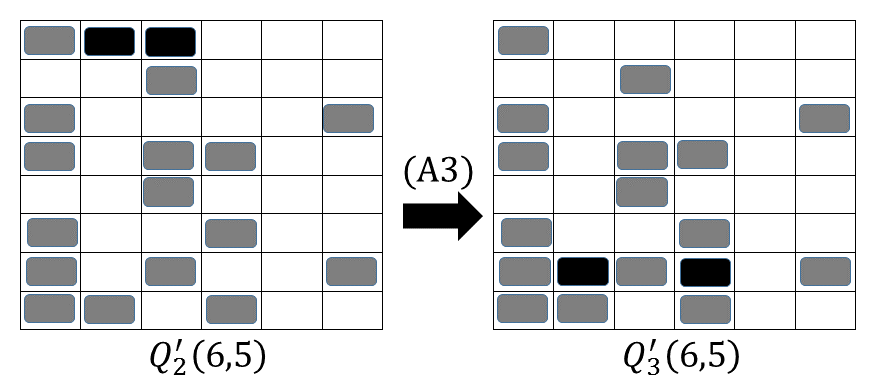}
\caption{$Q'_2(6,5)$ and $Q'_3(6,5)$} \label{Qex4}
\end{center}
\end{figure}

\end{example}

Let $\hat Q(n,k)$ be the configuration obtained by returning the upper two levels and first level of $Q(n,k)$ to $Q_3'(n,k)$. Note that $\hat Q(n,k)$ is in the same configuration as $Q$ up to the first and topmost levels. 

\begin{prop}\label{g(Q)leqg(n,k)even}
For even $n$ and any $Q$ we have $g(Q)\leq g(n,k)$. 
\end{prop}

\begin{proof}
Because the number of blocks in each level of $Q'(n,k)$ is $\frac{n}{2}$ and each piece is arranged with a 1-by-1 gap, the genus does not increase with the operations (A1), (A2), and (A3) of the algorithm for $Q'(n,k)$. Moreover, the increasing genus from $Q'_3(n,k)$ to $\hat Q(n,k)$ is not greater than the decreasing genus from $Q(n,k)$ to $Q'(n,k)$, which implies that $g(\hat Q(n,k))\leq g(Q(n,k))$. Note that $\hat Q(n,k)$ and $Q$ are in the same configuration up to the first and $(s(\hat Q(n,k))-1)$th levels. Because the number of blocks in the first level is $\frac{n}{2}$ and that in the $(s(\hat Q(n,k))-1)$th level is $n$, we have $g(Q)\leq g(\hat Q(n,k))$ and, hence, $g(Q)\leq g(Q(n,k))=g(n,k)$. 
\end{proof}

\subsection{Proof for odd $n$}
Now we assume that $n$ is odd. We will use same notations as in Subsection~\ref{Proof for even n}. 
Although we can apply the algorithm for odd $n$, the last argument in the proof of Proposition~\ref{g(Q)leqg(n,k)even} is not true in general. In fact, if $l$ (the number of blocks in the first level of $\hat Q(n,k)$) is less than $\frac{n+1}{2}$, then we have to estimate the genera of $\hat Q(n,k)$ and $Q$ more carefully. To do so, we introduce the following operations for $Q(n,k)$ (recall that we set $x=s(Q(n,k))$ for odd $n$): 

\begin{enumerate}
\item Apply the operation (I) $\frac{n-1}{2}$times to the $(x-2)$th level of $Q(n,k)$ by using all the blocks in the $x$th level. The resulting configuration is denoted by $Q^{(2)}(n,k)$. Note that $s(Q^{(2)}(n,k))=s(Q(n,k))-1$, $Q^{(2)}(n,k)$ has $n$ blocks in the $s(Q^{(2)}(n,k))$ and $(s(Q^{(2)}(n,k))-1)$th level, and $g(Q^{(2)}(n,k))<g(n,k)$. 
\item If the number of blocks in the first level of $Q$ is greater than the number of blocks in the first level of $Q^{(2)}(n,k)$ then, apply (I) to the first level of $Q^{(2)}(n,k)$ by using the blocks in the top level so that the resulting configuration $Q^{(3)}(n,k)$ has the same number of blocks as $Q$ in the first level. 
\end{enumerate}

\begin{prop}\label{g(Q)leqg(n,k)odd}
For odd $n$ and any $Q$ we have $g(Q)\leq g(n,k)$. 
\end{prop}
\begin{proof}
If $l$ is greater than or equal to the number of blocks in the first level of $Q$, then we can apply the same argument in the proof of Proposition~\ref{g(Q)leqg(n,k)even}. 

Otherwise, if $l$ is less than the number of blocks in the first level of $Q$, then we consider the configuration $Q^{(3)}(n,k)$ in the above operation. Note that because the decreasing genus in the first step [$Q(n,k)\rightarrow Q^{(2)}(n,k)$] is greater than its increasing genus in the second step [$Q^{(2)}(n,k) \rightarrow Q^{(3)}(n,k)$], we have $g(Q^{(3)}(n,k))< g(Q(n,k))$. By applying the algorithm, we can deform $Q^{(3)}(n,k)$ to $Q$ and show that $g(Q)\leq g(Q^{(3)}(n,k))$. So we have $g(Q)< g(Q(n,k))=g(n,k)$. 
\end{proof}


%

\end{document}